\documentclass[12pt]{article}
\usepackage[margin=1in]{geometry}
\DeclareMathAlphabet\mathbfcal{OMS}{cmsy}{b}{n}
\usepackage[utf8]{inputenc}
\usepackage[dvipsnames]{xcolor}
\usepackage{enumerate}
\usepackage[utf8]{inputenc}
\usepackage{amsfonts}
\usepackage{amsthm}
\usepackage{amsmath}
\usepackage{lscape}
\usepackage{mathrsfs}
\usepackage{graphicx}
\usepackage{tikz-cd}
\usepackage{calc}
\usepackage{float}
\usepackage{stmaryrd}
\usepackage{multirow}
\usepackage{lipsum}
\usepackage{amssymb}
\usepackage{mathrsfs}
\usepackage{mathtools}
\usepackage{mathdots}
\usepackage{fancyhdr}
\usepackage{tikz}
\usepackage{faktor}
\usepackage{setspace}
\usetikzlibrary{decorations.markings}
\newtheorem{theorem}{Theorem}
\theoremstyle{plain}
\usepackage{sectsty}
\usepackage{kpfonts}
\newtheorem{corollary}{Corollary}

\newtheorem{lemma}{Lemma}

\newtheorem{proposition}{Proposition}
\newtheorem{remark}{Remark}

\newtheorem*{proposition*}{Proposition}

\newcommand{\Hom}{{\mbox{Hom}}}

\newcommand{\End}{{\mbox{End}}}

\newcommand*{\QEDB}{\null\nobreak\hfill\ensuremath{\Box}}
\newtheorem*{observation*}{Observation}
\newtheorem*{theorem*}{Theorem}
\newtheorem*{claim*}{Claim}  
\usepackage{blindtext}
\usepackage{fancyhdr}
\title{Curvature, integrability, and the six sphere}
\author{Gabriella Clemente}
\date{}
\allowdisplaybreaks
\makeatletter
\renewcommand\tableofcontents{%
    \@starttoc{toc}%
}
\begin{document}
\maketitle

\begin{abstract}
This note is about the interplay between the almost-hermitian and Riemannian geometries of a manifold. These geometries can be seen to interact through curvature. The main result is an obstruction equation to the integrability of almost-complex structures orthogonal with respect to Riemannian metrics with constrained sectional curvature. Several geometric consequences ensue, such as a formula for the norm of the Levi-Civita covariant derivative of a hypothetical orthogonal complex structure. Our results lead to a partial recovery of the well-known fact that the round $6$-sphere $S^6$ is not hermitian. The partial proof is intrinsic in nature, and shows some level of promise when it comes to generalizing the non-complexity of the round $S^6$ result in new directions. 
\end{abstract} 

\section{Introduction}
This note is concerned with\\

\noindent
 \textbf{\underline{Question A}:} How does constraining the (various) curvatures of the Riemannian metric of an almost-hermitian structure affect integrability in high enough dimensions?\\ 
 
As demonstrated in \cite{ACOTI}, the curvature of torsion-free connections, in particular, the Levi-Civita one, is related to the integrability of almost-complex structures. In fact, in sufficiently high dimensions, constant curvature obstructs the existence of certain special complex structures at a very basic level (Theorem 1 \cite{ACOTI}). Question A can be tackled by an analysis of the obstruction equations from \cite{ACOTI}, but specialized to the Levi-Civita connection of a Riemannian metric with prescribed curvature and to metric compatible almost-complex structures. 

The best studied almost-hermitian example where integrability is obstructed is the round $6$-sphere $S^6.$ Hence, the round $S^6$ is a point of access to study Question A. The only spheres that can support an almost-complex structure are $S^2$ and $S^6,$ $S^2$ being, in addition, the complex manifold $\mathbb{CP}^1$ \cite{Borel}. It is an open problem to decide if $S^6$ is complex or not though there seems to be more evidence in support of the latter. $S^6$ cannot carry a complex structure that is orthogonal w.r.t.\ to the round metric \cite{LB} (see also the very clear exposition in \cite{Ferreira}). This theorem has an extension to metrics that are conformal to the round metric, and the nature of the proof is extrinsic \cite{Salamon}. The baseline argument is that a hypothetical hermitian complex structure on the conformally flat $S^6$ would produce a holomorphic embedding into a K{\"a}hler manifold, thereby forcing $S^6$ to be K{\"a}hler too, which is impossible as $H^2(S^6, \mathbb{R})=0.$ Although the use of such extrinsic methods is certainly efficient, it might be opaquing what is at the heart of the problem of (non-)integrability. Another extension of the theorem from \cite{LB} is to metrics nearby the round one \cite{Boris}.

This note falls short of re-proving LeBrun's theorem from \cite{LB}. However, the technique is intrinsic, it helps to address Question A, and shows some level of promise of being generalized. Specifically, the method makes uses of the curvature of a Riemannian metric in an explicit way to obtain an obstruction equation for the existence of an orthogonal complex structure. This equation leads to geometrical conclusions that are not too far away from the K{\"a}hler contradiction produced in \cite{LB}. It seems plausible then, that if one could re-establish Lebrun's result with this intrinsic approach, then one could also generalize the result in new directions. The generalization could happen by considering perturbations of the curvature obstruction equation coming from perturbations of the round metric. For additional comments on this, see the final section. 

The organization of this note is as follows: section 2 is a review of the original derivation of curvature obstructions from \cite{ACOTI}. Sections 3 -- 5 compute a refinement of the first curvature obstruction for the round $S^6,$ and use it to deduce various differential geometric facts (Corollary \ref{S6}). Section 5 also outlines a plan to study Question A through the special case of $S^6.$

\section{Obstructures}

This section is an overview of the parts of \cite{ACOTI} that will be relevant in the sequel. An informal presentation of this material can be found in \cite{Srni}. 

Let $M$ be a smooth manifold, and $\Omega^{\bullet} (M, T_M)=\bigoplus_{k=0}^n \Omega^k (M, T_M)$ be the space of tangent bundle valued differential forms. The spaces \[\Omega^{\bullet} \big(M, \End_{\mathbb{R}}(T_M)\big)=\bigoplus_{k\geq0} \Omega^k \big(M, \End_{\mathbb{R}}(T_M)\big), \mbox{ and } \Omega^{\bullet} \big(M, \bigwedge^{\bullet} {T_M}\big)=\bigoplus_{k\geq0} \bigoplus_{p+q=k}  \Omega^p \big(M, \bigwedge^q {T_M}\big)\] of endomorphism-, respectively polyvector-valued forms on $M$ are graded algebras for the following products: using the left self-action of $\End_{\mathbb{R}}(T_M),$ \[\cdot:\End_{\mathbb{R}}(T_M) \times T^*_M \otimes T_M \to T^*_M \otimes T_M,\] \[(S, f\otimes e)\mapsto S\cdot (f\otimes e):=f\otimes S(e),\] the product of $\alpha \in \Omega^k \big(M, \End_{\mathbb{R}}(T_M)\big)$ and $\beta \in \Omega^l \big(M, \End_{\mathbb{R}}(T_M)\big)$ is defined by the formula
\begin{equation*}
\begin{split}
(\alpha \wedge \beta)(X_1,\dots,X_{k+l})=\frac{1}{k!l!}\sum_{\sigma \in S_{k+l}} sign(\sigma) \alpha(X_{\sigma(1)},\dots,X_{\sigma(k)})\cdot \beta(X_{\sigma(k+1)},\dots,X_{\sigma(k+l)});
\end{split}
\end{equation*}
the product of $\gamma \in  \Omega^i \big(M, \bigwedge^j {T_M}\big)$ and $\theta \in \Omega^{k} \big(M, \bigwedge^{l} {T_M}\big)$ is given by
\begin{equation*}
\begin{split}
(\gamma \wedge \theta)(X_1,\dots,X_{i+k})=\frac{1}{2}\frac{1}{i!k!}\sum_{\sigma \in S_{i+k}} sign(\sigma) \gamma(X_{\sigma(1)},\dots,X_{\sigma(i)})\wedge \theta(X_{\sigma(i+1)},\dots,X_{\sigma(i+k)}).
\end{split}
\end{equation*}

Note that if $\gamma \in \Omega^i(M,T_M),$ and $\theta \in \Omega^k(M,T_M),$ then $\gamma \wedge \theta = (-1)^{ik+1} \theta \wedge \gamma.$ 

The graded algebras \[\Omega^{\bullet} \big(M, \End_{\mathbb{R}}(T_M)\big)=\bigoplus_{k\geq0} \Omega^k \big(M, \End_{\mathbb{R}}(T_M)\big), \mbox{ and } \Omega^{\bullet} \big(M, \bigwedge^{\bullet} {T_M}\big)=\bigoplus_{k\geq0} \bigoplus_{p+q=k}  \Omega^p \big(M, \bigwedge^q {T_M}\big)\] are certainly rings with the usual addition, and the space \[\Omega^{\bullet} (M, T_M)=\bigoplus_{k\geq 0}\Omega^k (M, T_M)\] can be viewed as both a left $\Omega^{\bullet} \big(M, \End_{\mathbb{R}}(T_M)\big)$-module and a right $\Omega^{\bullet} \big(M, \bigwedge^{\bullet} {T_M}\big)$-module. Let $\rho \in \Omega^s (M, T_M).$ The left action is given, for any $\alpha \in \Omega^k \big(M, \End_{\mathbb{R}}(T_M)\big),$ by

\begin{equation*}
\begin{split}
(\alpha \wedge \rho)(X_1,\dots,X_{k+s})=\frac{1}{k!s!}\sum_{\sigma \in S_{k+s}} sign(\sigma) \alpha(X_{\sigma(1)},\dots,X_{\sigma(k)})\big(\rho(X_{\sigma(k+1)},\dots,X_{\sigma(k+s)})\big),
\end{split}
\end{equation*}

The right action is given, for any $\gamma \in  \Omega^i \big(M, \bigwedge^j {T_M}\big),$ assuming that $s \geq j,$ by 
\begin{equation}\label{1}
\begin{split}
(\rho \wedge \gamma)(X_1,\dots,X_{s-j+i})=\frac{1}{(s-j)!i!}\sum_{\sigma \in S_{s-j+i}} sign(\sigma) \rho(X_{\sigma(1)},\dots,X_{\sigma(s-j)},\cdot,\dots,\cdot)\big(\gamma(X_{\sigma(s-j+1)},\dots,X_{\sigma(s-j+i)})\big).
\end{split}
\end{equation}
Otherwise, when $s<j,$ put $\rho \wedge \gamma=0.$ A closer look at formula \ref{1}: indeed, $\rho(X_{\sigma(1)},\dots,X_{\sigma(s-j)},\cdot,\dots,\cdot) \in \Omega^{j} (M, T_M).$ Since 
\[\Omega^{j} (M, T_M)=\Omega^0 \big(M, \bigwedge^j {T_M}^* \otimes T_M\big) \simeq \Omega^0 \Big(M, \Hom_{\mathbb{R}}\Big(\bigwedge^j {T_M},T_M\Big)\Big),\] one has that $\rho(X_{\sigma(1)},\dots,X_{\sigma(s-j)},\cdot,\dots,\cdot) \in \Omega^0 \Big(M, \Hom_{\mathbb{R}}\Big(\bigwedge^j {T_M},T_M\Big)\Big).$ Then, since \[\gamma(X_{\sigma(s-j+1)},\dots,X_{\sigma(s-j+i)})\ \in \bigwedge^j {T_M},\] \[\rho(X_{\sigma(1)},\dots,X_{\sigma(s-j)},\cdot,\dots,\cdot)\big(\gamma(X_{\sigma(s-j+1)},\dots,X_{\sigma(s-j+i)})\big) \in T_M.\] Hence, \[(\rho \wedge \gamma)(X_1,\dots,X_{s-j+i}) \in T_M\] as desired. 

Denote the space of of almost-complex structures on $M$ by $AC(M) \subset \Omega^1(M, T_M).$ Recall that $A \in AC(M)$ is integrable iff its Nijenhuis tensor vanishes identically, i.e.\ iff 

\begin{equation*}
\begin{split}
N_A (\zeta, \eta)&:=[A(\zeta), A(\eta)]-A([A(\zeta),\eta]+[\zeta,A(\eta)])-[\zeta,\eta]\\
&=0
\end{split}
\end{equation*}
for all vector fields $\zeta, \eta \in \mathfrak{X}(M)$ \cite{New}.

The curvature obstruction equations to the integrability of almost-complex structures from \cite{ACOTI} facilitate the probing of the almost-complex geometry of a manifold with Riemannian metrics of prescribed curvature. The idea behind this approach is straightforward: covariantly differentiate the Nijenhuis tensor $N_J$ of a $J \in AC(M).$ For $k \geq 1,$ the order $k$ covariant derivative of $N_J$ set equal to zero, $D^k N_J=0,$ can be taken as an obstruction equation. Certainly, if $D^k N_J \neq 0$ for some $k,$ then $J$ cannot be integrable. Call the left hand side of such an obstruction equation an \emph{obstructure};  e.g.\ the differential form $D^k N_J \in \Omega^{k+2}(M, T_M)$ is an obstructure. 

Observe that when $\dim_{\mathbb{R}}{M}=2,$ the integrability of almost-complex structures is unobstructed (i.e.\ for any $A \in AC(M^2),$ $N_A=0,$ and therefore $D^k N_A=0$ automatically for each $k$). Hence, the fact that the unit $S^2$ carries a constant curvature equal to $1$ Riemannian metric is inconsequential. Moreover, complex surfaces have been classified (see, for instance, \cite{EK}). It is for these reasons that Question A is posed for high dimensional $M.$ From this point on, it will be assumed that $\dim_{\mathbb{R}}{M} \geq 6.$

In \cite{ACOTI}, it was shown that the integrability of $A$ is equivalent to the $A$-invariance of $d^{\nabla}A$ (i.e.\ $d^{\nabla} A(A(X),A(Y))=d^{\nabla}A(X,Y)$), and $\nabla$ can be any symmetric connection.

\begin{lemma}{(Lemma 1 \cite{ACOTI})}\label{Lem1}
If $A \in AC(M),$ $\nabla$ is any torsion-free connection on $T_M,$ and $I^{\nabla}_A:=d^{\nabla} A \wedge (A \wedge A)-d^{\nabla} A,$ then $A$ is integrable iff $I^{\nabla}_A=0.$
\end{lemma}

The integrability form $I^{\nabla}_A$ is a function of $A$ and $d^{\nabla} A.$ So naturally, the covariant exterior derivative of $I^{\nabla}_A$ will depend on $A,$ $d^{\nabla} A,$ and $(d^{\nabla})^2 A=R^{\nabla} \wedge A.$ It should be evident that the kth covariant exterior derivative of $I^{\nabla}_A$ will depend on \[A, d^{\nabla} A, (d^{\nabla})^2 A=R^{\nabla} \wedge A,\dots, (d^{\nabla})^{k+1} A=\begin{cases}(R^{\nabla})^{\frac{k+1}{2}} \wedge A & \mbox{ if } k \mbox{ is odd}\\
(R^{\nabla})^{\frac{k}{2}} \wedge d^{\nabla} A & \mbox{ if } k \mbox{ is even.}
\end{cases}\]
See Lemma 4 \cite{ACOTI} for the computaton of $(d^{\nabla})^r A,$ $r>0.$ The integrability of $A$ implies that the kth obstructure $(d^{\nabla})^k I^{\nabla}_A$ vanishes for all $k \geq 1.$ Here, it will be enough to consider only the 1st obstructure:

\begin{lemma}{(Remarks 1 and 2 \cite{ACOTI})}\label{coe1}
\begin{equation}\label{eqo}
d^{\nabla} I^{\nabla}_A=(R^{\nabla} \wedge A)\wedge (A \wedge A)+2d^{\nabla} A \wedge (d^{\nabla} A \wedge A)-R^{\nabla} \wedge A.
\end{equation} 
\end{lemma}

See also Proposition 1 \cite{ACOTI} for the general, kth almost-complex obstructure. Equation \ref{eqo} can be further simplified, and this is covered in the next section. A succession of computational steps starting from this simplification will lead to an obstructure for the round $S^6.$ All results will be proven under the more general assumption of constant sectional curvature. On its own, this level of generality is immaterial. However, since the priority is to address Question A, it seems best to proceed in this fashion.

\section{The first curvature obstruction equation}
Throughout this section, with the exception of the last proposition, $\nabla$ is an arbitrary torsion-free connection on $T_M.$
\begin{lemma}\label{simple}
The 1st obstructure of $A \in AC(M)$ simplifies to \[d^{\nabla} I^{\nabla}_A=\frac{1}{2}R^{\nabla} \wedge A+2d^{\nabla} A\wedge (d^{\nabla} A \wedge A).\]
\end{lemma}

\begin{proof}
Let us first verify that \[(R^{\nabla} \wedge A) \wedge (A \wedge A)=\frac{3}{2}R^{\nabla} \wedge A.\] From Lemma \ref{coe1}, it is then immediate that \[d^{\nabla} I^{\nabla}_A=\frac{1}{2}R^{\nabla} \wedge A+2d^{\nabla} A\wedge (d^{\nabla} A \wedge A).\] Indeed, 
\begin{equation*}
\begin{split}
\big((R^{\nabla} \wedge A) \wedge (A \wedge A)\big)(X_1, X_2, X_3)&=\frac{1}{1!2!} \sum_{\sigma \in S_3} sign(\sigma) (R^{\nabla} \wedge A)(X_{\sigma(1)},\cdot,\cdot) \big((A \wedge A)(X_{\sigma(2)}, X_{\sigma(3)})\big) \\
&=\frac{1}{2} \sum_{\sigma \in S_3} sign(\sigma)(R^{\nabla} \wedge A)(X_{\sigma(1)},\cdot,\cdot) \big(A(X_{\sigma(2)})\wedge A(X_{\sigma(3)})\big) \\
&=\frac{1}{2} \sum_{\sigma \in S_3} sign(\sigma)(R^{\nabla} \wedge A)(X_{\sigma(1)},A(X_{\sigma(2)}),A(X_{\sigma(3)}))\\
&=\frac{1}{2} \sum_{\sigma \in S_3} sign(\sigma)\Big[\frac{1}{2!1!} \sum_{\sigma' \in S_3} sign(\sigma') \times \\
&R^{\nabla}(X_{\sigma'(\sigma(1))}, A(X_{\sigma'(\sigma(2))}))(A^2(X_{\sigma'(\sigma(3))}))\Big]\\
&=-\frac{1}{4} \sum_{\sigma \in S_3} sign(\sigma)\Big[\sum_{\sigma' \in S_3} sign(\sigma')R^{\nabla}(X_{\sigma'(\sigma(1))}, A(X_{\sigma'(\sigma(2))}))(X_{\sigma'(\sigma(3))})\Big]\\
&=-\frac{1}{4} \sum_{\sigma \in S_3} sign(\sigma)\Big[R^{\nabla}(X_{\sigma(1)}, A(X_{\sigma(2)}))(X_{\sigma(3)})+\\
&R^{\nabla}(X_{\sigma(2)}, A(X_{\sigma(3)}))(X_{\sigma(1)})+R^{\nabla}(X_{\sigma(3)}, A(X_{\sigma(1)}))(X_{\sigma(2)})\\
&-R^{\nabla}(X_{\sigma(2)}, A(X_{\sigma(1)}))(X_{\sigma(3)})-R^{\nabla}(X_{\sigma(1)}, A(X_{\sigma(3)}))(X_{\sigma(2)})\\
&-R^{\nabla}(X_{\sigma(3)}, A(X_{\sigma(2)}))(X_{\sigma(1)})\Big].
 \end{split}
\end{equation*}
To reach the second equality, note that
\begin{equation*}
\begin{split}
(A \wedge A)(X, Y)&=A(X)\wedge A(Y).
 \end{split}
\end{equation*}
Next, apply the $1^{st}$ Bianchi identity, \[R^{\nabla}(X,Y)(W)+R^{\nabla}(Y,W)(X)+R^{\nabla}(W,X)(Y)=0,\] to the last line. The conclusion is that
\begin{equation*}
\begin{split}
\big((R^{\nabla} \wedge A) \wedge (A \wedge A)\big)(X_1, X_2, X_3)&=-\frac{1}{4} \sum_{\sigma \in S_3} sign(\sigma)\Big[-R^{\nabla}(X_{\sigma(3)}, X_{\sigma(1)})(A(X_{\sigma(2)}))\\
&-R^{\nabla}(X_{\sigma(1)}, X_{\sigma(2)})(A(X_{\sigma(3)}))-R^{\nabla}(X_{\sigma(2)}, X_{\sigma(3)})(A(X_{\sigma(1)}))\Big]\\
&=\frac{1}{4} \sum_{\sigma \in S_3} sign(\sigma)\Big[R^{\nabla}(X_{\sigma(3)}, X_{\sigma(1)})(A(X_{\sigma(2)}))+\\
&R^{\nabla}(X_{\sigma(1)}, X_{\sigma(2)})(A(X_{\sigma(3)}))+R^{\nabla}(X_{\sigma(2)}, X_{\sigma(3)})(A(X_{\sigma(1)}))\Big].
\end{split}
\end{equation*}
Now, since 
\begin{equation*}
\begin{split}
(R^{\nabla} \wedge A)(X_1, X_2, X_3)&=\frac{1}{2!1!} \sum_{\sigma \in S_3} sign(\sigma)R^{\nabla}(X_{\sigma(1)}, X_{\sigma(2)})(A(X_{\sigma(3)}))\\
&=\frac{1}{2} \sum_{\sigma \in S_3} sign(\sigma)R^{\nabla}(X_{\sigma(1)}, X_{\sigma(2)})(A(X_{\sigma(3)})),
\end{split}
\end{equation*}
and likewise 
\begin{equation*}
\begin{split}
(R^{\nabla} \wedge A)(X_3, X_1, X_2)&=\frac{1}{2} \sum_{\sigma \in S_3} sign(\sigma)R^{\nabla}(X_{\sigma(3)}, X_{\sigma(1)})(A(X_{\sigma(2)}))
\end{split}
\end{equation*}
and
\begin{equation*}
\begin{split}
(R^{\nabla} \wedge A)(X_2, X_1, X_3)&=\frac{1}{2} \sum_{\sigma \in S_3} sign(\sigma)R^{\nabla}(X_{\sigma(2)}, X_{\sigma(1)})(A(X_{\sigma(3)})),
\end{split}
\end{equation*}
it follows from the total anti-symmetry of the $T_M$-valued $3$-form $R^{\nabla} \wedge A,$ that 
\begin{equation*}
\begin{split}
\big((R^{\nabla} \wedge A) \wedge (A \wedge A)\big)(X_1, X_2, X_3)&=\frac{1}{2} \Big[(R^{\nabla} \wedge A)(X_3, X_1, X_2)+(R^{\nabla} \wedge A)(X_1, X_2, X_3)+\\
&(R^{\nabla} \wedge A)(X_2, X_3, X_1) \Big]\\
&=\frac{3}{2} (R^{\nabla} \wedge A)(X_1, X_2, X_3).
\end{split}
\end{equation*}
\end{proof}

Lemma \ref{simple} suggests that the 1st curvature obstruction equation of $A \in AC(M)$ takes on the form

\begin{remark}{(Simplified 1st curvature obstruction equation)}\label{sl}
\[R^{\nabla} \wedge A+4d^{\nabla} A \wedge (d^{\nabla} A \wedge A)=0.\]
\end{remark}

Example 1 of \cite{ACOTI} verifies that 
\begin{lemma}\label{wed}
For any $\alpha \in \Omega^1 (M, T_M),$ $\beta \in \Omega^2 (M, T_M),$
\begin{equation*}
\begin{split}
\big(\beta \wedge (\beta \wedge \alpha)\big)(X_1, X_2, X_3)&=\frac{1}{2} \Big[\beta\big(\beta(X_1,X_2),\alpha(X_3)\big)+\beta\big(\beta(X_2,X_3),\alpha(X_1)\big)\\
&-\beta\big(\beta(X_1,X_3),\alpha(X_2)\big)\Big].
\end{split}
\end{equation*}
\end{lemma}

Let us expand the 1st curvature obstruction equation so as to make evident the passage between the almost-complex geometry of $M,$ which is a priori metric-independent, and the Riemannian geometry of $M.$

\begin{proposition}\label{metacx}
If $A \in AC(M)$ is integrable, then 
\begin{equation}\label{u1}
\begin{split}
&R^{\nabla}(X_1,X_2)(A(X_3))+R^{\nabla}(X_2,X_3)(A(X_1))-R^{\nabla}(X_1,X_3)(A(X_2))+2\Big[d^{\nabla} A\big(d^{\nabla}A (X_1,X_2),A(X_3)\big)+\\
&d^{\nabla} A\big(d^{\nabla} A(X_2,X_3),A(X_1)\big)-d^{\nabla} A\big(d^{\nabla} A(X_1,X_3),A(X_2)\big)\Big]=0.
\end{split}
\end{equation}
Moreover, when $\nabla$ is the Levi-Civita connection of a Riemannian metric $g$ on $M,$ then
\begin{equation}\label{u2}
\begin{split}
&Rm(X_1,X_2,A(X_3),X_4)+Rm(X_2,X_3,A(X_1),X_4))-Rm(X_1,X_3,A(X_2),X_4)+\\
&2\Big[g\Big(d^{\nabla} A\big(d^{\nabla}A (X_1,X_2),A(X_3)\big),X_4\Big)+g\Big(d^{\nabla} A\big(d^{\nabla}A (X_2,X_3),A(X_1)\big),X_4\Big)\\
&-g\Big(d^{\nabla} A\big(d^{\nabla}A (X_1,X_3),A(X_2)\big),X_4\Big)\Big]=0.
\end{split}
\end{equation}
\end{proposition}
\begin{proof}
Since \[(R^{\nabla} \wedge A)(X_1,X_2,X_3)=R^{\nabla}(X_1,X_2)(A(X_3))+R^{\nabla}(X_2,X_3)(A(X_1))-R^{\nabla}(X_1,X_3)(A(X_2)),\] equation \ref{u1} follows from Remark \ref{sl}, \[\big(R^{\nabla} \wedge A+4d^{\nabla} A \wedge (d^{\nabla} A \wedge A)\big)(X_1,X_2,X_3)=0,\] and Lemma \ref{wed} (with $\beta=d^{\nabla} A,$ $\alpha=A$). Equation \ref{u2} follows from writing out the left hand side of \[g \Big(\big(R^{\nabla} \wedge A+4d^{\nabla} A \wedge (d^{\nabla} A \wedge A)\big)(X_1,X_2,X_3),X_4\Big)=0.\]
\end{proof}

The first observation in this note, relating curvature and integrability is stated below.

\begin{corollary}\label{cor}
Let $(M,g)$ be an almost-complex, Riemannian manifold of constant sectional curvature $c.$ If $A \in AC(M)$ is integrable and $\nabla$ is the Levi-Civita connection of $g,$ then 
\begin{equation}\label{u3}
\begin{split}
&c\big(g(X_2,A(X_3))X_1- g(A(X_2),X_3)X_1-g(X_1,A(X_3))X_2+g(A(X_1),X_3)X_2+\\
&g(X_1,A(X_2))X_3- g(A(X_1),X_2)X_3\big)+2\Big[d^{\nabla} A\big(d^{\nabla}A (X_1,X_2),A(X_3)\big)+\\
&d^{\nabla} A\big(d^{\nabla} A(X_2,X_3),A(X_1)\big)-d^{\nabla} A\big(d^{\nabla} A(X_1,X_3),A(X_2)\big)\Big]=0,
\end{split}
\end{equation}
and
\begin{equation}\label{u4}
\begin{split}
&c\big(g(X_2,A(X_3))g(X_1,X_4)- g(A(X_2),X_3)g(X_1,X_4)-g(X_1,A(X_3))g(X_2,X_4)+g(A(X_1),X_3)g(X_2,X_4)+\\
&g(X_1,A(X_2))g(X_3,X_4)- g(A(X_1),X_2)g(X_3,X_4)\big)+2\Big[g\Big(d^{\nabla} A\big(d^{\nabla}A (X_1,X_2),A(X_3)\big),X_4\Big)+\\
&g\Big(d^{\nabla} A\big(d^{\nabla}A (X_2,X_3),A(X_1)\big),X_4\Big)-g\Big(d^{\nabla} A\big(d^{\nabla}A (X_1,X_3),A(X_2)\big),X_4\Big)\Big]=0.
\end{split}
\end{equation}
\end{corollary}
\begin{proof}
Equation \ref{u3} follows from the well-known formula for the curvature $R^{\nabla}$ of a constant sectional curvature manifold (see, for instance, Lemma 8.10 \cite{Lee}) together with equation \ref{u1}. Then, \ref{u4} is the $g$-inner product of \ref{u3} with $X_4.$
\end{proof}

\section{Interfacing the constant curvature and almost-hermitian geometries}
For any Riemannian metric $g$ on $M,$ let \[AC(M)_g:=\{J \in AC(M) \mid g(JX,JY)=g(X,Y), \forall X, Y \in \mathfrak{X}(M)\}\] be the space of all $g$-orthogonal almost-complex structures. For instance, if $(S^6, g)$ is the round $6$-sphere, then $AC(S^6)_g \neq \emptyset;$ e.g.\ the standard octonion almost-complex structure on $S^6$ is orthogonal. Let $\Omega_A:=g(A(\cdot),\cdot),$ for $A \in AC(M)_g,$ denote the fundamental $2$-form of the almost-hermitian manifold $(M,A,g).$

\begin{corollary}\label{cor}
If $(M,g)$ is an almost-complex Riemannian manifold of constant sectional curvature $c,$ $A \in AC(M)_g$ is integrable (i.e.\ $(M,A,g)$ is hermitian), and if $\nabla$ is the Levi-Civita connection of $g,$ then 
\begin{equation}\label{u5}
\begin{split}
&c\big(\Omega_A(X_1,X_3)X_2- \Omega_A(X_1,X_2)X_3-\Omega_A(X_2,X_3)X_1\big)+d^{\nabla} A\big(d^{\nabla} A(X_1,X_2),A(X_3)\big)+\\
&d^{\nabla} A\big(d^{\nabla} A(X_2,X_3),A(X_1)\big)-d^{\nabla} A\big(d^{\nabla} A(X_1,X_3),A(X_2)\big)=0,
\end{split}
\end{equation}
and
\begin{equation}\label{u6}
\begin{split}
&c\big(\Omega_A(X_1,X_3)g(X_2,X_4)- \Omega_A(X_1,X_2)g(X_3,X_4)-\Omega_A(X_2,X_3)g(X_1,X_4)\big)+\\
&g\Big(d^{\nabla} A\big(d^{\nabla}A (X_1,X_2),A(X_3)\big),X_4\Big)+g\Big(d^{\nabla} A\big(d^{\nabla}A (X_2,X_3),A(X_1)\big),X_4\Big)\\
&-g\Big(d^{\nabla} A\big(d^{\nabla}A (X_1,X_3),A(X_2)\big),X_4\Big)=0.
\end{split}
\end{equation}
\end{corollary}
\begin{proof}
Observe that since $A$ is $g$-orthogonal, 
\begin{equation}
\begin{split}
&g(X_2,A(X_3))X_1- g(A(X_2),X_3)X_1-g(X_1,A(X_3))X_2+g(A(X_1),X_3)X_2+\\
&g(X_1,A(X_2))X_3- g(A(X_1),X_2)X_3=2\big[g(X_2,A(X_3))X_1-g(X_1,A(X_3))X_2+\\
&g(X_1,A(X_2))X_3 \big]=2\big[g(A(X_1),X_3)X_2-g(A(X_1),X_2)X_3-g(A(X_2),X_3)X_1\big]=\\
&2\big[\Omega_A(X_1,X_3)X_2- \Omega_A(X_1,X_2)X_3-\Omega_A(X_2,X_3)X_1\big].
\end{split}
\end{equation}
So equation \ref{u3} becomes equation \ref{u5} (divide by $2$). Then, equation \ref{u6} follows from inner producting \ref{u5} and $X_4$ with $g.$
\end{proof}

\subsection{A first simplification}
In the spirit of \cite{Gray} (cf.\ section 2), one may regard a form $P \in \Omega^2(M,T_M)$ as a map $P:\bigwedge^2 T_M \to T_M,$ and thus write $P(X,Y)=P(X \wedge Y).$ As such, $P$ can be extended to a map $P:\bigwedge^{k+1} T_M \to \bigwedge^{k} T_M,$ for any $k >0,$ in the following way:

\[P(\zeta_1 \wedge \dots \wedge \zeta_{k+1})=\sum_{1 \leq i \leq j \leq k+1} (-1)^{i+j+1} P(\zeta_i \wedge \zeta_j)\wedge \zeta_1 \wedge \dots \wedge \widehat{\zeta_i} \wedge \dots \wedge \widehat{\zeta_j} \wedge \dots \wedge \zeta_{k+1}.\] Indeed, if $k=2,$ $P:\bigwedge^3 T_M \to \bigwedge^2 T_M,$ and 
\begin{equation*}
\begin{split}
P(\zeta_1 \wedge \zeta_2 \wedge \zeta_3)&=P(\zeta_1 \wedge \zeta_2) \wedge \zeta_3-P(\zeta_1 \wedge \zeta_3) \wedge \zeta_2+P(\zeta_2 \wedge \zeta_3) \wedge \zeta_1.
\end{split}
\end{equation*}
Thus, $P^2=P \circ P:\bigwedge^3 T_M \to T_M,$ and clearly 
\begin{equation*}
\begin{split}
P^2(\zeta_1 \wedge \zeta_2 \wedge \zeta_3)&=P(P(\zeta_1 \wedge \zeta_2) \wedge \zeta_3)-P(P(\zeta_1 \wedge \zeta_3) \wedge \zeta_2)+P(P(\zeta_2 \wedge \zeta_3) \wedge \zeta_1).
\end{split}
\end{equation*}
Now, let $A \in AC(M).$ Then, since $d^{\nabla} A \in \Omega^2(M,T_M),$ by the above discussion, $(d^{\nabla} A)^2:\bigwedge^3 T_M \to T_M,$ and so
\begin{equation*}
\begin{split}
(d^{\nabla} A)^2(X_1 \wedge X_2 \wedge X_3)=d^{\nabla} A\big(d^{\nabla} A(X_1,X_2),X_3\big)-d^{\nabla} A\big(d^{\nabla} A(X_1,X_3),X_2\big)+d^{\nabla} A\big(d^{\nabla} A(X_2,X_3),X_1\big).
\end{split}
\end{equation*}

\begin{proposition}\label{pfirst}
If $(M,g)$ is an almost-complex Riemannian manifold of constant sectional curvature $c,$ $A \in AC(M)_g$ is integrable, and if $\nabla$ is the Levi-Civita connection of $g,$ then
\begin{equation*}\label{oro}
\begin{split}
g\big((d^{\nabla}A)^2(X_1 \wedge X_2 \wedge X_3),X_4\big)+\frac{c}{2}(\Omega_A \wedge \Omega_A)(X_1,X_2,X_3,X_4)=0.
\end{split}
\end{equation*}
\end{proposition}

\begin{proof}
Set
\begin{equation*}\label{ultra}
\begin{split}
h(X_1,X_2,X_3)&:=c\big(\Omega_A(X_1,X_3)X_2- \Omega_A(X_1,X_2)X_3-\Omega_A(X_2,X_3)X_1\big)+d^{\nabla} A\big(d^{\nabla} A(X_1,X_2),A(X_3)\big)+\\
&d^{\nabla} A\big(d^{\nabla} A(X_2,X_3),A(X_1)\big)-d^{\nabla} A\big(d^{\nabla} A(X_1,X_3),A(X_2)\big).
\end{split}
\end{equation*}
By Corollary \ref{cor}, equation \ref{u5}, $h(X_1,X_2,X_3)=0,$ for all $X_1, X_2, X_3 \in \mathfrak{X}(M).$ Thus, in particular, 
\begin{equation*}
\begin{split}
0&=h(AX_1,AX_2,AX_3)\\
&=c\big(\Omega_A(X_1,X_3)(A(X_2))- \Omega_A(X_1,X_2)(A(X_3))-\Omega_A(X_2,X_3)(A(X_1))\big)-d^{\nabla} A\big(d^{\nabla} A(X_1,X_2),X_3\big)\\
&-d^{\nabla} A\big(d^{\nabla} A(X_2,X_3),X_1\big)+d^{\nabla} A\big(d^{\nabla} A(X_1,X_3),X_2\big).
\end{split}
\end{equation*}
After recognizing that \[(\Omega_A \wedge \Omega_A)(X_1,X_2,X_3,X_4)=2\big(\Omega(X_1,X_2)\Omega(X_3,X_4)-\Omega(X_1,X_3)\Omega(X_2,X_4)+\Omega(X_2,X_3)\Omega(X_1,X_4)\big),\] the claimed equation follows from expanding the left hand side of \[g\big(h(AX_1,AX_2,AX_3),X_4\big)=0,\] and the definition of $(d^{\nabla} A)^2.$ 
\end{proof}
This simplification, though more compact, can not be so easily manipulated. However, it can be taken a step forward in the following way. 

\subsection{A second simplification and its geometric consequences}
Let us introduce some convenient notation. For $g$ any Riemannian metric on $M,$ consider the bilinear mapping \[\wedge_g:\Omega^k(M,T_M)\otimes \Omega^l(M,T_M) \to \Omega^{k+l}(M),\] where \[(\alpha \wedge_g \beta)(X_,\dots,X_{k+l})=\frac{1}{k! l!} \sum_{\sigma \in S_{k+l}} sign(\sigma)g\big(\alpha(X_{\sigma(1)},\dots,X_{\sigma(k)}, \beta(X_{\sigma(k+1)},\dots,X_{\sigma(k+l)})\big).\] For example, 
\begin{equation}\label{why}
\begin{split}
(d^{\nabla} A \wedge_g d^{\nabla} A)(X_1,X_2,X_3,X_4)&=2\big[g\big(d^{\nabla} A(X_1,X_2), d^{\nabla} A(X_3,X_4)\big)-g\big(d^{\nabla} A(X_1,X_3), d^{\nabla} A(X_2,X_4)\big)+\\
&g\big(d^{\nabla} A(X_2,X_3), d^{\nabla} A(X_1,X_4)\big)\big].
\end{split}
\end{equation}

This final subsection is devoted to proving

\begin{theorem}\label{2sin}
If $(M,g)$ is an almost-complex Riemannian manifold of constant sectional curvature $c,$ $A \in AC(M)_g$ is integrable, and if $\nabla$ is the Levi-Civita connection of $g,$ then \[d^{\nabla} A \wedge_g d^{\nabla} A+2c \Omega_A \wedge \Omega_A=0.\]
\end{theorem}

The constant curvature obstruction equation $d^{\nabla} A \wedge_g d^{\nabla} A+2c \Omega_A \wedge \Omega_A=0$ is a structure equation of sorts. It is somewhat reminiscent of, for example, the Maurer-Cartan equation of the Mauer-Cartan form $\omega$ on a Lie group: $d\omega+\frac{1}{2}[\omega, \omega]=0.$

For any $A \in AC(M),$ define a tensor field $\Phi_A$ by the equation \[\Phi_A(X,Y,Z,W):=g\big((d^{\nabla}A)^2(X \wedge Y \wedge Z), W\big).\] Note that the curvature obstruction equation from Proposition \ref{pfirst} can be rewritten as 
\begin{equation}\label{can}
\begin{split}
\Phi_A+\frac{c}{2}\Omega_A \wedge \Omega_A=0.
\end{split}
\end{equation}
The proof of Theorem \ref{2sin} will require an analysis of $\Phi_A$ when $A$ is a hermitian complex structure. In this case, observe that indeed $\Phi_A$ is a $4$-form because $\Phi_A=-\frac{c}{2}\Omega_A \wedge \Omega_A \in \Omega^4(M).$ 

\begin{lemma}\label{care}
If $(M,g)$ is an almost-complex Riemannian manifold of constant sectional curvature, $A \in AC(M)_g$ is integrable, and if $\nabla$ is the Levi-Civita connection of $g,$ then 
\begin{equation}\label{creation}
\begin{split}
\Phi_A(X_1,X_2,X_3,X_4)&=g\big(d^{\nabla}A(X_1,X_2), (\nabla_{X_3} A)X_4\big)-g\big(d^{\nabla}A(X_1,X_3), (\nabla_{X_2} A)X_4\big)+\\
&g\big(d^{\nabla}A(X_2,X_3), (\nabla_{X_1} A)X_4\big).
\end{split}
\end{equation}
\end{lemma}

\begin{proof}
Put \[\phi_A(X_1,X_2,X_3,X_4):=g\big(d^{\nabla} A\big(d^{\nabla} A(X_1,X_2),X_3\big),X_4\big)\] so that \[\Phi_A(X_1,X_2,X_3,X_4)=\phi_A(X_1,X_2,X_3,X_4)-\phi_A(X_1,X_3,X_2,X_4)+\phi_A(X_2,X_3,X_1,X_4).\] Proposition \ref{pfirst} implies that $\Phi_A$ is an $A$-invariant $4$-form on $M.$ Now, from the $A$-invariance of $d^{\nabla}A$ (see Lemma \ref{Lem1}) together with orthogonality, it follows that

\begin{equation*}
\begin{split}
\phi_A(A(X_1),A(X_2),A(X_3),A(X_4))&=g\big(d^{\nabla} A\big(d^{\nabla} A(A(X_1),A(X_2)),A(X_3)\big),A(X_4)\big)\\
&=-g\big(A\big(d^{\nabla} A\big(d^{\nabla} A(X_1,X_2),A(X_3)\big)\big), X_4\big).
\end{split}
\end{equation*}
Observe that 
\begin{equation}\label{leap1}
(\nabla_X A)(AY)=-A((\nabla_X A)Y) 
\end{equation}
This is actually true for any almost-complex structure. The hermitian nature of $A$ implies that
\begin{equation}\label{leap2}
(\nabla_{A(X)} A)Y=A((\nabla_X A)Y);
\end{equation}
see, for example, Lemma 5.4 in \cite{Moroianu}. Then,

\begin{equation*}
\begin{split}
\phi_A(A(X_1),A(X_2),A(X_3),A(X_4))&=-g\big(A\big(d^{\nabla} A\big(d^{\nabla} A(X_1,X_2),A(X_3)\big)\big), X_4\big)\\
&=-g\Big(A\Big(\big(\nabla_{d^{\nabla} A(X_1,X_2)} A\big)(A(X_3))\Big)-\big(\nabla_{A(X_3)} A\big)(d^{\nabla} A(X_1,X_2))\Big), X_4\Big)\\
&=-g\Big(\big(\nabla_{d^{\nabla} A(X_1,X_2)} A\big)X_3+ \big(\nabla_{X_3} A\big)(d^{\nabla} A(X_1,X_2)), X_4\Big)\\
&=-g\Big(d^{\nabla} A\big(d^{\nabla} A(X_1,X_2),X_3\big)+ 2\big(\nabla_{X_3} A\big)(d^{\nabla} A(X_1,X_2)), X_4\Big)\\
&=-g\big(d^{\nabla} A\big(d^{\nabla} A(X_1,X_2),X_3\big),X_4\big)-2g\big(\big(\nabla_{X_3} A\big)(d^{\nabla} A(X_1,X_2)), X_4\big)\\
&=-\phi_A(X_1,X_2,X_3,X_4)-2g\big(\big(\nabla_{X_3} A\big)(d^{\nabla} A(X_1,X_2)), X_4\big).
\end{split}
\end{equation*}
Similarly, by the above equation, 
\begin{equation*}
\begin{split}
-\phi_A(A(X_1),A(X_3),A(X_2),A(X_4))&=\phi_A(X_1,X_3,X_2,X_4)+2g\big(\big(\nabla_{X_2} A\big)(d^{\nabla} A(X_1,X_3)), X_4\big),
\end{split}
\end{equation*}
and 
\begin{equation*}
\begin{split}
\phi_A(A(X_2),A(X_3),A(X_1),A(X_4))&=-\phi_A(X_2,X_3,X_1,X_4)-2g\big(\big(\nabla_{X_1} A\big)(d^{\nabla} A(X_2,X_3)), X_4\big).
\end{split}
\end{equation*}
Therefore, it must be that 
\begin{equation*}
\begin{split}
\Phi_A(A(X_1),A(X_2),A(X_3),A(X_4))&=\phi_A(A(X_1),A(X_2),A(X_3),A(X_4))-\phi_A(A(X_1),A(X_3),A(X_2),A(X_4))+\\
&\phi_A(A(X_2),A(X_3),A(X_1),A(X_4))\\
&=-\Big[\phi_A(X_1,X_2,X_3,X_4)+2g\big(\big(\nabla_{X_3} A\big)(d^{\nabla} A(X_1,X_2)), X_4\big) \Big]+\\
&\Big[\phi_A(X_1,X_3,X_2,X_4)+2g\big(\big(\nabla_{X_2} A\big)(d^{\nabla} A(X_1,X_3)), X_4\big)\Big]\\
&-\Big[\phi_A(X_2,X_3,X_1,X_4)+2g\big(\big(\nabla_{X_1} A\big)(d^{\nabla} A(X_2,X_3)), X_4\big)\Big]\\
&=-\Phi_A(X_1,X_2,X_3,X_4)-2\Big[g\big(\big(\nabla_{X_3} A\big)(d^{\nabla} A(X_1,X_2)), X_4\big)\\
&-g\big(\big(\nabla_{X_2} A\big)(d^{\nabla} A(X_1,X_3)), X_4\big)+g\big(\big(\nabla_{X_1} A\big)(d^{\nabla} A(X_2,X_3)), X_4\big)\Big]\\
&=\Phi_A(X_1,X_2,X_3,X_4).
\end{split}
\end{equation*}
The last 2 lines in the previous equation then say that
\begin{equation*}
\begin{split}
\Phi_A(X_1,X_2,X_3,X_4)&=-\Big[g\big(\big(\nabla_{X_3} A\big)(d^{\nabla} A(X_1,X_2)), X_4\big)\\
&-g\big(\big(\nabla_{X_2} A\big)(d^{\nabla} A(X_1,X_3)), X_4\big)+g\big(\big(\nabla_{X_1} A\big)(d^{\nabla} A(X_2,X_3)), X_4\big)\Big].
\end{split}
\end{equation*}
And now, since for any vector fields $X,Y,Z \in \mathfrak{X}(M),$ one has that \[g((\nabla_X A)Y,Z)=-g((\nabla_X A)Z, Y),\] it is easy to see that
\begin{equation*}
\begin{split}
\Phi_A(X_1,X_2,X_3,X_4)&=g\big(d^{\nabla}A(X_1,X_2), (\nabla_{X_3} A)X_4\big)-g\big(d^{\nabla}A(X_1,X_3), (\nabla_{X_2} A)X_4\big)+\\
&g\big(d^{\nabla}A(X_2,X_3), (\nabla_{X_1} A)X_4\big).
\end{split}
\end{equation*}
\end{proof}

\begin{theorem}\label{aura}
If $(M,g)$ is an almost-complex Riemannian manifold of constant sectional curvature, $A \in AC(M)_g$ is integrable, and if $\nabla$ is the Levi-Civita connection of $g,$ then for all vector fields $X,Y,Z \in \mathfrak{X}(M),$

\[g\big(d^{\nabla}A(X,Y),(\nabla_Z A)Z\big)=0.\]
\end{theorem}

\begin{proof}
By Lemma \ref{care}, the $4$-form $\Phi_A$ satisfies, for any $X,Y,Z \in \mathfrak{X}(M),$

\begin{equation*}
\begin{split}
0&=\Phi_A(X,Y,Z,Z)\\
&=g\big(d^{\nabla}A(X,Y), (\nabla_Z A)Z\big)-g\big(d^{\nabla}A(X,Z), (\nabla_Y A)Z\big)+\\
&g\big(d^{\nabla}A(Y,Z), (\nabla_X A)Z\big)\\
&=g\big(d^{\nabla}A(X,Y), (\nabla_Z A)Z\big)+g\big((\nabla_Z A)X,(\nabla_Y A)Z \big)-g\big((\nabla_Z A)Y,(\nabla_X A)Z \big).
\end{split}
\end{equation*}
So, 
\begin{equation}\label{eros}
g\big(d^{\nabla}A(X,Y), (\nabla_Z A)Z\big)=g\big((\nabla_Z A)Y,(\nabla_X A)Z \big)-g\big((\nabla_Z A)X,(\nabla_Y A)Z \big).
\end{equation}
The integrability of $A$ implies that the left hand side of equation \ref{eros} is $A$-invariant in $X, Y;$ i.e.\ 
\begin{equation}\label{miracle}
g\big(d^{\nabla}A(A(X),A(Y)), (\nabla_Z A)Z\big)=g\big(d^{\nabla}A(X,Y), (\nabla_Z A)Z\big).
\end{equation}
But the right hand side of equation \ref{eros} is $A$-anti-invariant in $X, Y.$ Indeed, identities \ref{leap1} and \ref{leap2} give that
\begin{equation}\label{from}
\begin{split}
g\big((\nabla_Z A)(A(Y)),(\nabla_{A(X)} A)Z \big)-g\big((\nabla_Z A)(A(X)),(\nabla_{A(Y)} A)Z \big)&=\\
-\big[g\big((\nabla_Z A)Y,(\nabla_X A)Z \big)-g\big((\nabla_Z A)X,(\nabla_Y A)Z \big)\big]
\end{split}
\end{equation}
And from equations \ref{eros}, \ref{miracle}, and \ref{from}, one gets that
\begin{equation*}
\begin{split}
g\big(d^{\nabla}A(X,Y), (\nabla_Z A)Z\big)&=g\big(d^{\nabla}A(A(X),A(Y)), (\nabla_Z A)Z\big)\\
&=g\big((\nabla_Z A)(A(Y)),(\nabla_{A(X)} A)Z \big)-g\big((\nabla_Z A)(A(X)),(\nabla_{A(Y)} A)Z \big)\\
&=-\big[g\big((\nabla_Z A)Y,(\nabla_X A)Z \big)-g\big((\nabla_Z A)X,(\nabla_Y A)Z \big)\big]\\
&=-g\big(d^{\nabla}A(X,Y), (\nabla_Z A)Z\big),
\end{split}
\end{equation*}
thereby confirming that \[g\big(d^{\nabla}A(X,Y), (\nabla_Z A)Z\big)=0.\]
\end{proof}

\begin{corollary}\label{or}
Let $(M,A,g)$ be a hermitian manifold, and assume that $g$ has constant sectional curvature. If $\nabla$ is the Levi-Civita connection, then \[\Omega_A((\nabla_X A)X, (\nabla_Y A)Y)=0.\]
\end{corollary}

\begin{proof}
Since $d^{\nabla} A(X,A(X))=-2((\nabla_X A)X),$ the claim follows directly from Theorem \ref{aura}.
\end{proof}

\noindent
\emph{Proof of Theorem \ref{2sin}.} From Theorem \ref{aura}, \[g\big(d^{\nabla}A(X,Y), (\nabla_{Z+W} A)(Z+W)\big)=0,\] and this holds iff 
\begin{equation*}
g\big(d^{\nabla}A(X,Y), (\nabla_Z A)W\big)=-g\big(d^{\nabla}A(X,Y), (\nabla_W A)Z\big).
\end{equation*}
Now, apply this observation to Lemma \ref{care} to find that
\begin{equation*}
\begin{split}
\Phi_A(X,Y,Z,W)&=g\big(d^{\nabla}A(X,Y), (\nabla_Z A)W\big)-g\big(d^{\nabla}A(X,Z), (\nabla_Y A)W\big)+\\
&g\big(d^{\nabla}A(Y,Z), (\nabla_X A)W\big)\\
&=-g\big(d^{\nabla}A(X,Y), (\nabla_W A)Z\big)+g\big(d^{\nabla}A(X,Z), (\nabla_W A)Y\big)\\
&-g\big(d^{\nabla}A(Y,Z), (\nabla_W A)X\big).
\end{split}
\end{equation*}
By adding the last 2 lines in the above equation, it becomes clear that 

\begin{equation}\label{claridad}
\begin{split}
\Phi_A(X,Y,Z,W)&=\frac{1}{2}\big[g\big(d^{\nabla} A(X,Y), d^{\nabla} A(Z,W)\big)-g\big(d^{\nabla} A(X,Z), d^{\nabla} A(Y,W)\big)+\\
&g\big(d^{\nabla} A(Y,Z), d^{\nabla} A(X,W)\big)\big].
\end{split}
\end{equation}
And now, from equations \ref{can} and \ref{claridad}, it follows that
\begin{equation*}
\begin{split}
0=&\Phi_A(X,Y,Z,W)+\frac{c}{2} (\Omega_A \wedge \Omega_A)(X,Y,Z,W)\\
&=\frac{1}{2}\big[g\big(d^{\nabla} A(X,Y), d^{\nabla} A(Z,W)\big)-g\big(d^{\nabla} A(X,Z), d^{\nabla} A(Y,W)\big)+\\
&g\big(d^{\nabla} A(Y,Z), d^{\nabla} A(X,W)\big)\big]+\frac{c}{2} (\Omega_A \wedge \Omega_A)(X,Y,Z,W),
\end{split}
\end{equation*}
or equivalently (cf.\ equation \ref{why}),
\begin{equation*}
\begin{split}
0&=2\big[g\big(d^{\nabla} A(X,Y), d^{\nabla} A(Z,W)\big)-g\big(d^{\nabla} A(X,Z), d^{\nabla} A(Y,W)\big)+\\
&g\big(d^{\nabla} A(Y,Z), d^{\nabla} A(X,W)\big)\big]+2c(\Omega_A \wedge \Omega_A)(X,Y,Z,W)\\
&=(d^{\nabla} A \wedge_g d^{\nabla} A)(X,Y,Z,W)+2c(\Omega_A \wedge \Omega_A)(X,Y,Z,W).
\end{split}
\end{equation*}
\QEDB

\begin{remark}\label{norm}
Let $(M,A,g)$ be an almost-hermitian manifold of constant sectional curvature equal to $c,$ and let $\nabla$ be the Levi-Civita connection. If $A$ is integrable, then \[\|(\nabla_X A)Y\|^2_g=g\big((\nabla_X A)X, (\nabla_Y A)Y\big)-\frac{c}{2}\big(\Omega_A(X,Y)^2-\|X\|^2_g \|Y\|^2_g +g(X,Y)^2\big).\]
\end{remark}

\begin{proof}
Consider equation \ref{creation}, and recall that \[d^{\nabla}A(X,A(X))=-2A\big((\nabla_X A)X\big).\] Then,  
\begin{equation*}
\begin{split}
\Phi_A(X,Y,A(X),A(Y))&=g\big(d^{\nabla}A(X,Y), (\nabla_{A(X)} A)(A(Y))\big)-g\big(d^{\nabla}A(X,A(X)), (\nabla_Y A)(A(Y))\big)+\\
&g\big(d^{\nabla}A(Y,A(X)), (\nabla_X A)(A(Y))\big)\\
&=g\big((\nabla_X A)Y- (\nabla_Y A)X, (\nabla_X A)Y\big)-2g\big((\nabla_X A)X,(\nabla_Y A)Y \big)+\\
&g\big((\nabla_X A)Y+(\nabla_Y A)X, (\nabla_X A)Y\big)\\
&=2\big[[\|(\nabla_X A)Y\|^2_g - g\big((\nabla_X A)X,(\nabla_Y A)Y \big)\big].
\end{split}
\end{equation*}
Moreover, since \[\frac{c}{2}(\Omega_A \wedge \Omega_A)(X,Y,A(X),A(Y))=c\big(\Omega_A(X,Y)^2-\|X\|^2_g \|Y\|^2_g +g(X,Y)^2\big),\]
it follows that
\begin{equation*}
\begin{split}
0&=\Phi_A (X,Y,A(X),A(Y))+\frac{c}{2}(\Omega_A \wedge \Omega_A)(X,Y,A(X),A(Y))\\
&=2\big[[\|(\nabla_X A)Y\|^2_g - g\big((\nabla_X A)X,(\nabla_Y A)Y \big)\big]+\\
&c\big(\Omega_A(X,Y)^2-\|X\|^2_g \|Y\|^2_g +g(X,Y)^2\big).
\end{split}
\end{equation*}
Therefore, \[\|(\nabla_X A)Y\|^2_g=g\big((\nabla_X A)X, (\nabla_Y A)Y\big)-\frac{c}{2}\big(\Omega_A(X,Y)^2-\|X\|^2_g \|Y\|^2_g +g(X,Y)^2\big).\]
\end{proof}

\section{Concluding remarks}
In the above, it seemed more transparent to work under the general assumption that the metric has constant sectional curvature equal to $c,$ where $c$ is any real number. However, the case of interest is surely $c=1,$ corresponding to the round metric of a sphere. Directly below is a summary of the observations made so far, but specialized to the relevant case.

\begin{corollary}\label{S6}
Let $g$ be the round metric on the unit $S^6,$ and denote its Levi-Civita connection by $\nabla.$ A hypothetical $g$-orthogonal complex structure $A$ on $S^6$ would have to solve the curvature obstruction equation \[d^{\nabla} A \wedge_{g}  d^{\nabla} A+2 \Omega_A \wedge \Omega_A=0,\] where $\Omega_A$ is the associated fundamental $2$-form. Moreover, such an orthogonal complex structure $A$ would need to have the following properties:
\begin{enumerate}
\item $\|(\nabla_X A)Y\|^2_g=g\big((\nabla_X A)X, (\nabla_Y A)Y\big)-\frac{1}{2}\big(\Omega_A(X,Y)^2-\|X\|^2_g \|Y\|^2_g +g(X,Y)^2\big),$ 

\item $g\big(d^{\nabla}A(X,Y),(\nabla_Z A)Z\big)=0,$ and 

\item $\Omega_A((\nabla_X A)X, (\nabla_Y A)Y)=0.$
\end{enumerate}
In particular, the norm of the bilinear form determined by the Levi-Civita covariant derivative of $A$ would have to be symmetric: $\|(\nabla_X A)Y\|_g=\|(\nabla_Y A)X\|_g.$
\end{corollary}

Let us now comment further on the remarks appearing in the introduction. In order to complete the (re)proof of the non-complexity of the round $S^6,$ it would suffice to show that the actual bilinear form $B_A(X,Y):=(\nabla_X A)Y$ is either symmetric (indicating K{\"a}hler), or equivalently, skew-symmetric (nearly K{\"a}hler). Meaning, the reproof would be finished if it could be shown that $(\nabla_X A)Y=\pm(\nabla_Y A)X.$ Recall that $B_A$ being symmetric means that $d^{\nabla} A=0,$ which is equivalent to the K{\"a}hler condition $\nabla A=0$ (see section 3 of \cite{CSACPS}). Notice as well that a nearly K{\"a}hler structure $A$ is integrable iff it is K{\"a}hler. Indeed, the integrability form of such an $A$ is $I^{\nabla}_A(X,Y)=-4(\nabla_X A)Y.$ The latter formula can be deduced from equation \ref{leap1} together with the skew-symmetry of $B_A.$ This is not to say that the K{\"a}hler contradiction actually follows from the constant curvature obstruction equation. Regardless, it could be worthwhile to look at Corollary \ref{S6} through the Gray-Hervella classification \cite{GH}. Perhaps, that point of view could assist in further unravelling the geometric meaning of the constant curvature obstruction equation, and its consequences.

Despite these shortcomings, Corollary \ref{S6} may be viewed as a partial recovery of LeBrun's theorem \cite{LB}. If the gaps in the suggested approach could be sealed, the method could turn out to be useful for generalizing LeBrun's theorem in new directions. 

A target generalization would be to $S^6$ equipped with a Riemannian metric $g'=g+\epsilon$ that is a small deformation of the round metric $g$ with controlled curvature. One would expect that the associated (sectional curvature) obstructure is given by a perturbation of the round metric obstructure, $d^{\nabla} A \wedge_{g}  d^{\nabla} A+2 \Omega_A \wedge \Omega_A.$ And then, one might try to establish that the perturbed obstruction equation implies that an integrable $A \in AC(S^6)_{g'}$ must be K{\"a}hler, and hence cannot exist.

But it might be possible to tackle the $S^6$ problem more generally. Any almost-complex structure on $S^6$ can be turned into an almost-hermitian structure in the following way. Let $A \in AC(S^6)$ be any almost-complex structure. Put $g_A:=\frac{1}{2}(g(\cdot, \cdot)+g(A \cdot, A \cdot)),$ where $g$ is the round metric. Note that $A \in AC(S^6)_{g_A}$ (i.e.\ $(M, A, g_A)$ is an almost-hermitian manifold). Due to the defining formula of $g_A,$ one would again expect for any associated curvature obstructure, such as the sectional curvature one, to be a perturbation of that obstructure for $g.$ Similarly, one could try to investigate when the integrability of $A$ implies, via the associated curvature obstruction equation, that $A$ should, in addition, be (nearly) K{\"a}hler. This is a contradiction. If this approach worked, it would potentially give a proof that $S^6$ cannot be complex, and the proof would be using intrinsic tools only.\\

Comments are welcomed.

\section*{Acknowledgement}
I wrote this note in Paris in an autonomous way. However, I produced the current version while being supported by the DMS Fellowship of Masaryk University, and the grant project GACR 22-15012J. I am grateful for the discussions that took place during the Srni 43rd Winter School, Geometry and Physics. I would also like to acknowledge Uwe Semmelmann for his feedback.

\noindent 
G.\ Clemente

\noindent
Department of Mathematics and Statistics, Masaryk University, Building 8, Kotlarska 2, 61137 Brno, Czech Republic

\noindent
e-mail: clemente6171@gmail.com
\end{document}